\documentclass[11pt]{amsart}
\usepackage{verbatim, graphicx, rotating}
\usepackage{mdwlist}
\usepackage{enumerate}
\usepackage{url}
\usepackage{amsthm}
\usepackage{hyperref}

\usepackage {amssymb}

\input xy
\xyoption{all}
\input epsf

\newlength{\tabwidth}
\newlength{\tabheight}
\newlength{\tabrule}
\newlength{\tabwidthx}
\newlength{\tabheightx}

\def\gentabbox#1#2#3#4{\vbox to \tabheight{\setlength{\tabrule}{#3}%
  \setlength{\tabwidthx}{#1\tabwidth}\addtolength{\tabwidthx}{\tabrule}%

\setlength{\tabheightx}{#2\tabheight}\addtolength{\tabheightx}{-\tabheight}%
  \hbox to #1\tabwidth{%
    \hspace{-0.5\tabrule}\rule{\tabrule}{#2\tabheight}\hspace{-\tabrule}%
    \vbox to #2\tabheight{\hsize=\tabwidthx%
      \vspace{-0.5\tabrule}\hrule width\tabwidthx height\tabrule%
      \vspace{-0.5\tabrule}\vfil%
      \hbox to \tabwidthx{\hss#4\hss}%
        \vfil\vspace{-0.5\tabrule}%
      \hrule width\tabwidthx height\tabrule\vspace{-0.5\tabrule}}%
    \hspace{-\tabrule}\rule{\tabrule}{#2\tabheight}\hspace{-0.5\tabrule}}%
  \vspace{-\tabheightx}}}
\def\genblankbox#1#2{\vbox to \tabheight{\vfil\hbox to
#1\tabwidth{\hfil}}}
\def\tabbox#1#2#3{\gentabbox{#1}{#2}{0.4pt}{\strut #3}}

\catcode`\:=13 \catcode`\.=13 \catcode`\;=13 
\catcode`\>=13 \catcode`\^=13
\def:#1\\{\hbox{$#1$}}
\def.#1{\tabbox{1}{1}{$#1$}}
\def>#1{\tabbox{2}{1}{$#1$}}
\def^#1{\tabbox{1}{2}{$#1$}}
\def;{\genblankbox{1}{1}\relax}
\catcode`\:=12 \catcode`\.=12 \catcode`\;=12 
\catcode`\>=12 \catcode`\^=7

\newcommand{\field}{\mathbb}
\newcommand{\liealgebra}{\mathfrak}
\newcommand{\la}{\liealgebra}

\newcommand{\C}{{\field C}}
\newcommand{\R}{{\field R}}
\newcommand{\Z}{{\field Z}}
\newcommand{\N}{{\field N}}

\renewcommand{\b}{\liealgebra b}

\newcommand{\n}{{\la n}}

\newcommand{\ga}{\alpha}

\newtheorem{prop}{Proposition}[section]

\newtheorem{lemma}[prop]{Lemma}
\newtheorem{theorem}[prop]{Theorem}

\theoremstyle{definition}

\newtheorem{remark}[prop]{Remark}
\newtheorem{example}[prop]{Example}

\newtheorem*{question}{Question}

\newtheorem{definition}[prop]{Definition}

\newcommand{\frt}{\mathfrak{t}}

\newcommand{\bbN}{\mathbb{N}}

\newcommand{\bbP}{\mathbb{P}}

\begin{document}
\title[$K$-orbits on $G/B$ and Schubert structure constants]
{$GL(p,\C) \times GL(q,\C)$-orbit closures on the flag variety and Schubert structure constants for $(p,q)$-pairs}

\author{Benjamin J. Wyser}
\date{\today}


\begin{abstract}
We give positive combinatorial descriptions of Schubert structure constants $c_{u,v}^w$ for the full flag variety in type $A_{n-1}$ when $u$ and $v$ form what we refer to as a ``$(p,q)$-pair" ($p+q=n$).  The key observation is that a certain subset of the $GL(p,\C) \times GL(q,\C)$-orbit closures on the flag variety (those satisfying an easily stated pattern avoidance condition) are Richardson varieties.  The result on structure constants follows when one combines this observation with a theorem of Brion concerning intersection numbers of spherical subgroup orbit closures and Schubert varieties.
\end{abstract}

\maketitle

\section{Introduction}
Let $G = GL(n,\C)$, with $B,B^- \subseteq G$ the Borel subgroups consisting of upper-triangular and lower-triangular matrices, respectively.  For each permutation $w \in S_n$, there exists a \textit{Schubert class} $S_w = [\overline{B^-wB/B}] \in H^*(G/B)$.  It is well-known that the classes $\{S_w\}_{w \in S_n}$ form a $\Z$-basis for $H^*(G/B)$.  As such, for any $u,v \in S_n$, we have
\[ S_u \cdot S_v = \displaystyle\sum_{w \in S_n} c_{u,v}^w S_w \]
in $H^*(G/B)$, for uniquely determined integers $c_{u,v}^w$.  These integers are the \textit{Schubert structure constants}.

The structure constants are known to be non-negative for geometric reasons, and are readily computable.  However, it is a long-standing open problem to give a \textit{positive} (i.e. subtraction-free) formula for $c_{u,v}^w$ in terms of the permutations $u$, $v$, and $w$.

There are numerous partial results which give positive formulas for structure constants $c_{u,v}^w$ in special cases.  Perhaps most notably, when $u,v$ are ``Grassmannian" permutations (each having a unique descent in the same place), the Schubert classes $S_u,S_v \in H^*(G/B)$ are pulled back from Schubert classes in the cohomology of a Grassmannian, and their products are determined by the classical Littlewood-Richardson rule, or by the equivalent ``puzzle rule" of \cite{Knutson-Tao-Woodward-04}.  Other examples include
\begin{itemize}
	\item Monk's rule (\cite{Monk-59}), which describes structure constants $c_{u,s_i}^w$ with $u \in S_n$ any permutation, and $s_i = (i,i+1)$ a simple transposition;
	\item An analogue of Pieri's rule for Grassmannians, which generalizes Monk's rule.  The formula determines $c_{u,v}^w$ when $u \in S_n$ is any permutation, and $v$ is a Grassmannian permutation of a certain ``shape" (\cite{Sottile-96});
	\item A rule due to M. Kogan (\cite{Kogan-01}) which describes $c_{u,v}^w$ when $u$ is a Grassmannian permutation with unique descent at $k$, and $v$ is a permutation all of whose descents are in positions at most $k$.  (Note that this generalizes the Littlewood-Richardson rule mentioned above.)
	\item A rule due to I. Coskun (\cite{Coskun-09}), which gives a positive description of structure constants in the cohomology ring of a two-step flag variety in terms of ``Mondrian tableaux".  (A manuscript on an extension of this rule to arbitrary partial flag varieties, currently available on I. Coskun's webpage, is described there as ``under revision"\footnote{Per \url{http://www.math.uic.edu/~coskun/}, as of August 31, 2011, the paper is described as follows:  ``Currently under revision. This is a preliminary version of a Littlewood-Richardson rule for arbitrary partial flag varieties. Any comments, corrections and suggestions are welcome."}.)
\end{itemize}

The above does not purport to be a comprehensive list of everything that is currently known about Schubert structure constants, but simply to give some indication of the types of cases currently understood.

The main result of this paper is a rule which gives a positive description of structure constants $c_{u,v}^w$ in another special case --- namely, the case where $u,v$ are what we call a \textit{$(p,q)$-pair} ($p+q=n$).  (See Definition \ref{def:pq_pair}.)  The rule is multiplicity-free, meaning that all such $c_{u,v}^w$ are either $0$ or $1$.  The case $0$ or $1$ is detected by computing a monoidal action of $w$ on the ``$(p,q)$-clan" associated to the permutations $u$ and $v$.  This boils down, in the end, to an elementary combinatorial check.  The statement of the rule is the content of Theorem \ref{thm:structure_constants}, the main result of this paper.

The significance of $(p,q)$-clans is that they parametrize the orbits of $GL(p,\C) \times GL(q,\C)$ on $G/B$.  In fact, it was the author's study of these orbits and their closures which led to the discovery of the aforementioned rule.  The key observation, which is the content of Theorem \ref{thm:k-orbits-richardson-varieties}, is that the closures of a number of these orbits (namely, those whose clans satisfy an easily stated pattern avoidance condition) are actually \textit{Richardson varieties}:  intersections of Schubert varieties with opposite Schubert varieties.  The rule for structure constants follows when one combines this observation with a theorem of Brion (Theorem \ref{thm:brion}), which describes the class of such an orbit closure as a sum of Schubert cycles in terms of paths in the weak order graph.  (See Section \ref{sect:weak_order}.)

$GL(p,\C) \times GL(q,\C)$ is an example of what is known as a ``symmetric subgroup" of $GL(n,\C)$:
\begin{definition}
Let $G$ be a complex reductive algebraic group, with $\theta$ an involution of $G$ (i.e. an automorphism whose square is the identity).  The fixed point subgroup $G^{\theta}$ is called a \textbf{symmetric subgroup}.
\end{definition}    
Such a subgroup is typically denoted $K$.  It is a fact (\cite{Matsuki-79}) that any symmetric subgroup $K$ acts on the flag variety for $G$ with finitely many orbits.  The geometry of these orbits and their closures plays a central role in the theory of Harish-Chandra modules for a certain real form $G_{\R}$ of $G$ --- namely, the real form containing a maximal compact subgroup $K_{\R}$ whose complexification is $K$.  (In the case of $GL(p,\C) \times GL(q,\C)$, the corresponding real form of $GL(n,\C)$ is $U(p,q)$.)  As such, it seems possible that Theorem \ref{thm:k-orbits-richardson-varieties} could also be of representation-theoretic interest.

The results of this paper can be pushed a bit farther to obtain similar rules for structure constants in types $C$ and $D$.  We briefly describe this.  Let $G' = Sp(2n,\C)$ or $SO(2n,\C)$, and let $X'$ be the flag variety for $G'$.  When $K = GL(n,\C) \times GL(n,\C) \subseteq GL(2n,\C)$ is intersected with $G'$, the result is a symmetric subgroup of $G'$, isomorphic to $K'=GL(n,\C)$ in each case.  Moreover, the intersection of any $K$-orbit with $X'$, if non-empty, is a single $K'$-orbit on $X'$.  In particular, intersecting any $K$-orbit closure which coincides with a Richardson variety with $X'$ gives a $K'$-orbit closure on $X'$ which is a Richardson variety in $X'$.  Because Brion's theorem applies to the class of \textit{any} symmetric subgroup orbit closure (better yet, to the class of any \textit{spherical} subgroup orbit closure) in \textit{any} flag variety, one may apply it again in these settings to obtain analogous rules for structure constants in types $CD$.  For more details, see \cite{Wyser-11b}.

We now describe the organization of this paper.  Section 1 is the introduction.  In Section 2, we cover some preliminaries, primarily concerning basic facts and definitions related to Schubert varieties, opposite Schubert varieties, and Richardson varieties, which we will use throughout the paper.  Section 3 gives a detailed description of the full and weak Bruhat orders on symmetric subgroup orbit closures, and includes the statement of Brion's theorem.  In Sections 4 and 5, we focus attention on the combinatorics of our main example, $K = GL(p,\C) \times GL(q,\C)$.  The material of Section 4 is a summary of known results gathered from other references, while the material of Section 5 is mostly new.  In Section 6, we use the combinatorics developed in the previous two sections to prove that certain of the $K$-orbit closures coincide with Richardson varieties.  In Section 7, we use this observation, along with Theorem \ref{thm:brion}, to prove our main result on structure constants.  Finally, we conclude by posing a natural question in Section 8 --- namely, are there other spherical subgroups of $GL(n,\C)$ some of whose orbit closures happen to coincide with Richardson varieties?

The author thanks William A. Graham, his research advisor at the University of Georgia, for his guidance throughout the author's doctoral thesis project, which led to the discovery of the results presented here.  The author further thanks Professor Graham for his assistance and advice in reading and editing early versions of this manuscript.  We also thank Allen Knutson and Alexander Yong for helpful email exchanges.

\section{Preliminaries}\label{sect:prelims}
\subsection{Notation}
We first establish some standard notation:

$[n]$ shall denote the set $\{1,\hdots,n\}$.

The long element of a Weyl group $W$ will always be denoted $w_0$.  When $W=S_n$, $w_0$ is the permutation $n\ n-1\ \hdots \ 2\ 1$.

We denote the variety of complete flags on $\C^n$ by $Fl(\C^n)$.  A complete flag $F_{\bullet}$ is a filtration
\[ F_0 \subset F_1 \subset \hdots \subset F_{n-1} \subset F_n \]
with $\dim(F_i) = i$ for $i=0,1,\hdots,n$.  (In particular, $F_0 = \{0\}$, and $F_n = \C^n$.)  In the notation of the introduction, $Fl(\C^n)$ is isomorphic to $G/B$.

For each $i \in [n]$, denote by $E_i$ the linear span $\C \left\langle e_1,\hdots,e_i \right\rangle$ of the \textit{first} $i$ standard basis vectors, and by $\widetilde{E_i}$ the linear span $\C \left\langle e_n,e_{n-1},\hdots,e_{n-i+1} \right\rangle$ of the \textit{last} $i$ standard basis vectors.

\subsection{Permutations and the Bruhat order}
For any permutation $w \in S_n$, and for each $(i,j) \in [n] \times [n]$, define
\[ r_w(i,j) = \#\{k \leq i \ \vert \ w(k) \leq j\}. \]

\begin{definition}
We refer to the matrix $(e_{i,j}) = (r_w(i,j))$ as the \textbf{rank matrix} for the permutation $w$.
\end{definition}

If one imagines a rank matrix to have a ``$0$th row" consisting only of $0$'s, rank matrices for permutations have the property that for each $i \in [n]$, row $i$ matches row $i-1$ until a certain point $n_i$ at which there is a ``jump" --- i.e. $r_w(i,n_i) = r_w(i-1,n_i) + 1$.  Beyond $n_i$, all entries in row $i$ are $1$ greater than the corresponding entries in row $i-1$.  Moreover, the ``jumps" $n_i$ are clearly just the values of the permutation:  $n_i = w(i)$.  Thus a permutation is completely determined by its rank matrix.

We give two equivalent definitions of the Bruhat order on $S_n$.  We will make use of both definitions later.  That these two definitions are equivalent to each other (and to the ``usual" definition of the Bruhat order) is standard --- see \cite{Deodhar-77} or \cite[\S 10.5]{Fulton-YoungTableaux}.
\begin{definition}\label{def:bruhat-1}
The \textbf{Bruhat order} on $S_n$ is the partial order defined as follows:  $u \leq v$ if and only if
\[ r_u(i,j) \geq r_v(i,j) \text{ for all } i,j. \]
\end{definition}

\begin{definition}\label{def:bruhat-2}
Here is an alternative definition of the Bruhat order on $S_n$:  $u \leq v$ if and only if for any $i \in [n]$, when $\{u(1),\hdots,u(i)\}$ and $\{v(1),\hdots,v(i)\}$ are each arranged in ascending order, each element of the first set is less than or equal to the corresponding element of the second set.
\end{definition}

\subsection{Schubert varieties, opposite Schubert varieties, and Richardson varieties}\label{sect:schubert_defs}
In this section, we recall some very basic definitions and facts regarding Schubert varieties, opposite Schubert varieties, and Richardson varieties.  The facts given here are all standard, and can be found in, e.g., \cite{Fulton-YoungTableaux} and/or \cite{Brion-05}.

Let $G$ be a reductive, complex algebraic group, with $B,B^- \subseteq G$ opposed Borel subgroups, and $T = B \cap B^-$ a maximal torus of $G$.  Let $W = N_G(T)/T$ be the Weyl group for $(G,T)$, and let $G/B$ be the flag variety.  By the ``type $A$ case", we will mean the following standard setup:
\begin{itemize}
	\item $G=GL(n,\C)$
	\item $B=$ upper-triangular matrices
	\item $B^-=$ lower-triangular matrices
	\item $T=$ diagonal matrices
	\item $W =S_n$
\end{itemize}
In this setup, $G/B \cong Fl(\C^n)$.

\begin{definition}\label{def:schubert_cells_varieties}
For each $w \in W$, the \textbf{Schubert cell} $C_w$ is defined to be $BwB/B$.  In the type $A$ case, for $w \in S_n$, we have
\[ C_w = \{ F_{\bullet} \in Fl(\C^n) \ \vert \ \dim(F_i \cap E_j) = r_w(i,j) \ \forall i,j\}. \]

The Zariski closure of the cell $C_w$ is a \textbf{Schubert variety}, and will be denoted $X_w$.  $X_w$ is an irreducible subvariety of $G/B$ of complex dimension $l(w)$.  
\end{definition}

Next, we define opposite Schubert cells and opposite Schubert varieties.  

\begin{definition}\label{def:opposite_cells_varieties}
For each $w \in W$, the \textbf{opposite Schubert cell} $C^w$ is defined to be $B^-wB/B$.  In the type $A$ case, for $w \in S_n$, we have
\[ C^w := \{ F_{\bullet} \in Fl(\C^n) \ \vert \ \dim(F_i \cap \widetilde{E_j}) = r_{w_0w}(i,j) \ \forall i,j\}. \]

The Zariski closure of the opposite cell $C^w$ is an \textbf{opposite Schubert variety}, and will be denoted $X^w$.  It is an irreducible subvariety of $G/B$ of complex \textit{codimension} $l(w)$.
\end{definition}

\begin{definition}
For each $w \in W$, the \textbf{Schubert class} $S_w$ is the (Poincar\'e dual to the) fundamental class of the opposite Schubert variety $X^w$.  That is, $S_w = [X^w]$.  Note that $S_w \in H^{2l(w)}(G/B)$.
\end{definition}

The fundamental classes of Schubert varieties and opposite Schubert varieties in $H^*(G/B)$ are related in the following way:
\begin{equation}\label{eqn:classes_of_schubert_and_opposite_schubert}
	[X^v] = [X_{w_0v}].
\end{equation}

We turn our attention now to intersections of Schubert varieties with opposite Schubert varieties.
\begin{prop}\label{prop:basic_richardson_varieties}
For $u,v \in W$, $X_u^v := X_u \cap X^v$ is non-empty if and only if $u \geq v$.  In this event, the intersection $X_u \cap X^v$ is proper and reduced, and has dimension $l(u) - l(v)$.  Moreover, $C_u^v := C_u \cap C^v$ is open and dense in $X_u^v$.
\end{prop}

\begin{definition}
For $u,v \in W$ with $u \geq v$, the intersection $X_u^v$ is called a \textbf{Richardson variety}.
\end{definition}

In light of equation (\ref{eqn:classes_of_schubert_and_opposite_schubert}) and Proposition \ref{prop:basic_richardson_varieties}, we have the following in $H^*(G/B)$:
\begin{equation}\label{eqn:classes_of_richardson_varieties}
[X_u^v] = [X_u] \cdot [X^v] = S_{w_0u} \cdot S_v.
\end{equation}

Finally, we state the following fact which relates the containment order on Richardson varieties (or the closure order on subsets $C_u^v$) to Bruhat intervals:
\begin{prop}\label{prop:richardson_closure_order}
For $u \geq v$, $u' \geq v'$, $X_{u'}^{v'} \subseteq X_u^v$ if and only if $u \geq u' \geq v' \geq v$.
\end{prop}

\section{The weak and full closure orders on symmetric subgroup orbits}\label{sect:weak_order}
Let $G$ be any complex, reductive algebraic group, with $\theta: G \rightarrow G$ an involution, and $K=G^{\theta}$ the corresponding symmetric subgroup.  Let $T \subseteq B$ be a $\theta$-stable maximal torus and Borel subgroup, respectively.

The orbits of $K$ on $G/B$ are partially ordered by closure containment:  $Q_1 \leq Q_2 \Leftrightarrow \overline{Q_1} \subseteq \overline{Q_2}$.  We shall refer to this order as the ``closure order" or the ``full closure order".  A weaker order, which we call the ``weak order" or ``weak closure order", can be defined as follows:  For any simple root $\ga \in \Delta(G,T)$, let $P_{\alpha}$ denote the standard minimal parabolic subgroup of type $\ga$, and let
\[ \pi_{\ga}: G/B \rightarrow G/P_{\ga} \]
be the natural projection.  This is a locally trivial fiber bundle with fiber $P_{\ga}/B \cong \bbP^1$.  Given any $K$-orbit $Q$, one may consider the set $Z_{\ga}(Q):=\pi_{\ga}^{-1}(\pi_{\ga}(Q))$.  Because the map $\pi_{\ga}$ is $K$-equivariant, $Z_{\ga}(Q)$ is stable under $K$.  Assuming $K$ is connected, $Z_{\ga}(Q)$ is also irreducible, and so it contains a dense $K$-orbit.  (In the event that $K$ is disconnected, one notes that the component group of $K$ acts transitively on the irreducible components of $Z_{\ga}(Q)$, and from this it again follows that $Z_{\ga}(Q)$ has a dense $K$-orbit.  We will not need this more general fact.)  We denote this dense orbit by $s_{\ga} \cdot Q$.

If $\dim(\pi_{\ga}(Q)) < \dim(Q)$, then $s_{\ga} \cdot Q = Q$.  However, if $\dim(\pi_{\ga}(Q)) = \dim(Q)$, then $s_{\ga} \cdot Q = Q'$ for some $Q' \neq Q$ with $\dim(Q') = \dim(Q) + 1$.

\begin{definition}
The \textbf{weak closure order} (or simply the \textbf{weak order}) is the partial order on $K$-orbits generated by relations of the form $Q \prec Q' \Leftrightarrow Q' = s_{\ga} \cdot Q \neq Q$.
\end{definition}

Note that we may just as well speak of the weak ordering on orbit \textit{closures}.  Supposing that $Y,Y'$ are the closures of orbits $Q,Q'$, respectively, we say that $Y' = s_{\ga} \cdot Y$ if and only if $Q' = s_{\ga} \cdot Q$, if and only if $Y' = \pi_{\ga}^{-1}(\pi_{\ga}(Y))$.

If $Y' = s_{\ga} \cdot Y \neq Y$, one can consider the degree of the restricted morphism $\pi_{\ga}|_Y$ over its image.  It turns out that this degree is always either $1$ or $2$, as we now explain.  If $Y' = s_{\ga} \cdot Y$, then either:
\begin{enumerate}
	\item $\ga$ is ``complex" for $Q$; or
	\item $\ga$ is ``non-compact imaginary" for $Q$.
\end{enumerate}

The latter case breaks up into two subcases, known as ``type I" and ``type II".  These cases are differentiated by the $K$-orbit structure on the set $Z_{\ga}(Q)$ defined above.  In the ``type I" case, $Z_{\ga}(Q)$ is comprised of the dense orbit $Q'$, the orbit $Q$, and one other orbit $s_{\ga} \times Q$.  Here, $\times$ denotes the ``cross-action" of the Weyl group $W$ on $K \backslash G/B$, defined as
\[ w \times (K \cdot gB) = K \cdot gw^{-1}B. \]
In the ``type II" case, $Z_{\ga}(Q)$ is comprised simply of the dense orbit $Q'$ and the orbit $Q$, and in fact, $s_{\ga} \times Q = Q$.  In particular, if one knows that $\ga$ is non-compact imaginary for $Q$, then whether it is type I or type II depends only on whether $Q$ is fixed by the cross-action of $s_{\ga}$.

With all of this said, the result is as follows:
\begin{prop}
Suppose that $Y' = s_{\ga} \cdot Y \neq Y$.  If $\ga$ is complex or non-compact imaginary type I for $Q$, then $\pi_{\ga}|_Y$ is birational over its image.  If $\ga$ is non-compact imaginary type II for $Q$, then $\pi_{\ga}|_Y$ has degree 2 over its image.
\end{prop}
For a proof of this fact, and for more details on the definitions of complex, non-compact imaginary, etc., the reader may consult \cite{Richardson-Springer-90}.

In \cite{Brion-01}, the poset graph for the weak order on $K \backslash G/B$ is endowed with additional data as follows:  Whenever $Y' = s_{\ga} \cdot Y$ and $\ga$ is non-compact imaginary type II, $Y$ and $Y'$ are connected by a double edge.  In the complex or non-compact imaginary type I cases, simple edges are used.  Each edge, whether simple or double, is also labelled with the appropriate simple root $\ga$, or perhaps an index $i$ if $s_{\ga} = s_{\ga_i}$ for some ordering on the simple roots.

If $w \in W$, with $s_{i_1} \hdots s_{i_k}$ a reduced expression for $w$, set 
\[ w \cdot Y = s_{i_1} \cdot (s_{i_2} \cdot \hdots (s_{i_k} \cdot Y) \hdots ). \]
This is well-defined, independent of the choice of reduced expression for $w$, and defines an action of a certain monoid $M(W)$ on the set of $K$-orbit closures (\cite{Richardson-Springer-90}).  As a set, the monoid $M(W)$ is comprised of elements $m(w)$, one for each $w \in W$.  The multiplication on $M(W)$ is defined inductively by
\[ m(s)m(w) = 
\begin{cases}
	m(sw) & \text{ if $l(sw) > l(w)$}, \\
	m(w) & \text{ otherwise.}
\end{cases} \]

(We will use the notation $w \cdot Y$, as opposed to $m(w) \cdot Y$, to indicate this action, with the understanding that this defines an action of $M(W)$, and not of $W$.)

Suppose that $Y$ is a $K$-orbit closure on $G/B$ of codimension $d$.  Define the following subset of $W$:
\[ W(Y) := \{ w \in W \ \vert \ w \cdot Y = G/B \text{ and } l(w) = d\}. \]
(Note that in this definition, ``$G/B$" refers to the closure of the dense, open orbit.)  Elements of $W(Y)$ are precisely those $w$ such that there is a path connecting $Y$ to the top vertex of the weak order graph, the product of whose edge labels is $w$.  For any $w \in W(Y)$, denote by $D(w)$ the number of double edges in such a path.  (Although there may be more than one, any such path has the same number of double edges, so $D(w)$ \textit{is} well-defined.  See \cite[Lemma 5]{Brion-01}.)

We now recall a theorem of Brion, which we will ultimately use to obtain a positive rule for structure constants $c_{u,v}^w$ for certain types of pairs $(u,v)$.
\begin{theorem}[\cite{Brion-01}]\label{thm:brion}
With notation as above, in $H^*(G/B)$, the fundamental class of $Y$ is expressed in the Schubert basis as follows:
\[ [Y] = \displaystyle\sum_{w \in W(Y)} 2^{D(w)} S_w. \]
\end{theorem}

\section{Example:  $K = GL(p,\C) \times GL(q,\C)$}\label{sect:pq_example}
We now describe in detail our main example, the symmetric pair 
\[ (G,K)=(GL(n,\C),GL(p,\C) \times GL(q,\C)) \]
for any $p,q$ with $p+q=n$.  Let $\theta = \text{int}(I_{p,q})$, where
\[  I_{p,q} := 
\begin{pmatrix}
I_p & 0 \\
0 & -I_q \end{pmatrix}. \]

One checks easily that $K:=G^{\theta} \cong GL(p,\C) \times GL(q,\C)$, embedded diagonally as follows:
\[ K = \left\{
\left[
\begin{array}{cc}
K_{11} & 0 \\
0 & K_{22} \end{array}
\right] \in GL(n,\C)
\ \vert \ 
\begin{array}{c}
K_{11} \in GL(p,\C) \\
K_{22} \in GL(q,\C) \end{array}
\right\} .\]

The finitely many $K$-orbits on $G/B$ are parametrized by \textit{$(p,q)$-clans}, as described in, e.g., \cite{Matsuki-Oshima-90,Yamamoto-97,McGovern-Trapa-09}.  We recall this parametrization in detail.  
\begin{definition}
A \textbf{$(p,q)$-clan} is a string of $n=p+q$ symbols, each of which is a $+$, a $-$, or a natural number.  The string must satisfy the following two properties:
\begin{enumerate}
	\item Every natural number which appears must appear exactly twice in the string.
	\item The difference in the number of plus signs and the number of minus signs in the string must be $p-q$.  (If $q > p$, then there should be $q-p$ more minus signs than plus signs.)
\end{enumerate}
\end{definition}

We only consider such strings up to an equivalence which says, essentially, that it is the positions of matching natural numbers, rather than the actual values of the numbers, which determine the clan.  So, for instance, the clans $(1,2,1,2)$, $(2,1,2,1)$, and $(5,7,5,7)$ are all the same, since they all have matching natural numbers in positions $1$ and $3$, and also in positions $2$ and $4$.  On the other hand, $(1,2,2,1)$ is a different clan, since it has matching natural numbers in positions $1$ and $4$, and in positions $2$ and $3$.

The set of $(p,q)$-clans is in bijection with the set of $K$-orbits on $G/B$.  Moreover, given a clan $\gamma$, the orbit $Q_{\gamma}$ admits an explicit linear algebraic description in terms of the combinatorics of $\gamma$.  Let $E_p = \C \left\langle e_1,\hdots,e_p \right\rangle$ be the span of the first $p$ standard basis vectors, and let $\widetilde{E_q} = \C \left\langle e_{p+1},\hdots,e_n \right\rangle$ be the span of the last $q$ standard basis vectors.  Let $\pi: \C^n \rightarrow E_p$ be the projection onto $E_p$.

For any clan $\gamma=(c_1,\hdots,c_n)$, and for any $i,j$ with $i<j$, define the following quantities:
\begin{enumerate}
	\item $\gamma(i; +) = $ the total number of plus signs and pairs of equal natural numbers occurring among $(c_1,\hdots,c_i)$;
	\item $\gamma(i; -) = $ the total number of minus signs and pairs of equal natural numbers occurring among $(c_1,\hdots,c_i)$; and
	\item $\gamma(i; j) = $ the number of pairs of equal natural numbers $c_s = c_t \in \N$ with $s \leq i < j < t$.
\end{enumerate}

For example, for the $(2,2)$-clan $\gamma=(1,+,1,-)$, we have that 
\begin{enumerate}
	\item $\gamma(i; +) = 0,1,2,2$ for $i=1,2,3,4$;
	\item $\gamma(i; -) = 0,0,1,2$ for $i=1,2,3,4$; and
	\item $\gamma(i;j) = 1,0,0,0,0,0$ for $(i,j) = (1,2), (1,3), (1,4), (2,3), (2,4), (3,4)$.
\end{enumerate}

With all of this notation defined, we have the following theorem on $K$-orbits on $G/B$:
\begin{theorem}\label{thm:orbit_description}
Suppose $p+q=n$.  For a $(p,q)$-clan $\gamma$, define $Q_{\gamma}$ to be the set of all flags $F_{\bullet}$ having the following three properties for all $i,j$ ($i<j$):
\begin{enumerate}
	\item $\dim(F_i \cap E_p) = \gamma(i; +)$
	\item $\dim(F_i \cap \widetilde{E_q}) = \gamma(i; -)$
	\item $\dim(\pi(F_i) + F_j) = j + \gamma(i; j)$
\end{enumerate}

For each $(p,q)$-clan $\gamma$, $Q_{\gamma}$ is nonempty and stable under $K$.  In fact, $Q_{\gamma}$ is a single $K$-orbit on $G/B$.

Conversely, every $K$-orbit on $G/B$ is of the form $Q_{\gamma}$ for some $(p,q)$-clan $\gamma$.  Hence the association $\gamma \mapsto Q_{\gamma}$ defines a bijection between the set of all $(p,q)$-clans and the set of $K$-orbits on $G/B$.
\end{theorem}

\begin{remark}
The parametrization of $K$-orbits on $G/B$ by $(p,q)$-clans was described first in \cite{Matsuki-Oshima-90}.  In that paper, no proof of the correctness of the parametrization is given, and the above linear algebraic description of $Q_{\gamma}$ does not appear.  Both the proof and the explicit description of $Q_{\gamma}$ appear in \cite{Yamamoto-97}.
\end{remark}

We also record here for later use a formula for the dimension of the $K$-orbit $Q_{\gamma}$ in terms of the clan $\gamma$.  First define the ``length" of a clan $\gamma$ to be 
\begin{equation}
\label{e:length}
l(\gamma) = \sum_{c_i=c_j \in \bbN, i<j}\left( j-i-
\#\{k \in \bbN \; | \; c_s = c_t = k \text{ for some }
s < i<t<j \} \right ).
\end{equation}
Then
\[
\dim(Q_{\gamma}) = 
d(K) +  l(\gamma),
\]
where $d(K)$ is the dimension of the flag variety for $K$, namely
$\frac12(p(p-1) + q(q-1))$.

Next, we describe the weak order on $K \backslash G/B$ in terms of this parametrization (\cite{Yamamoto-97,McGovern-Trapa-09}).  Let $\frt$ be the Cartan subalgebra of $\text{Lie}(G) = \mathfrak{gl}(n,\C)$ consisting of diagonal matrices.  Let $x_i$ ($i=1,\hdots,n$) be coordinates on $\frt$, with
\[ x_i(\text{diag}(a_1,\hdots,a_n)) = a_i. \]
The simple roots are of the form $\ga_i = x_i - x_{i+1}$ ($i=1,\hdots,n-1$).  The root $\ga_i$ is complex for the orbit $Q_{\gamma}$ corresponding to $\gamma=(c_1,\hdots,c_n)$ if and only if $(c_i,c_{i+1})$ satisfy one of the following:
\begin{enumerate}
	\item $c_i$ is a sign, $c_{i+1}$ is a number, and the mate of $c_{i+1}$ occurs to the right of $c_{i+1}$;
	\item $c_i$ is a number, $c_{i+1}$ is a sign, and the mate of $c_i$ occurs to the left of $c_i$; or
	\item $c_i$ and $c_{i+1}$ are unequal natural numbers, and the mate of $c_i$ occurs to the left of the mate of $c_{i+1}$.
\end{enumerate}

In these cases, the orbit $s_{\ga_i} \cdot Q_{\gamma}$ is $Q_{\gamma'}$, where the clan $\gamma'$ is obtained from $\gamma$ by interchanging $c_i$ and $c_{i+1}$.

As examples of (1), (2), and (3) above, when $p=q=2$, we have
\begin{enumerate}
	\item $s_{\ga_1} \cdot (+,1,-,1) = (1,+,-,1)$; 
	\item $s_{\ga_2} \cdot (1,1,+,-) = (1,+,1,-)$; and
	\item $s_{\ga_2} \cdot (1,1,2,2) = (1,2,1,2)$.
\end{enumerate}

On the other hand, $\ga_i$ is non-compact imaginary for $Q_{\gamma}$ if and only if $(c_i,c_{i+1})$ are opposite signs.  In this case, $s_{\ga_i} \cdot Q_{\gamma} = Q_{\gamma''}$, where $\gamma''$ is obtained from $\gamma$ by replacing the signs in positions $(i,i+1)$ by matching natural numbers.  So, for instance, when $p=q=2$, $s_{\ga_2} \cdot (+,+,-,-) = (+,1,1,-)$.

The cross-action of $w \in S_n$ on any clan $\gamma$ (more correctly, on the orbit $Q_{\gamma}$) is the obvious one, by permutation of the characters of $\gamma$.  In particular, when $\ga_i$ is non-compact imaginary for $\gamma$, the cross-action of $s_{\ga_i}$ interchanges the opposite signs in positions $(i,i+1)$.  Thus for a non-compact imaginary root $\ga_i$, $s_{\ga_i} \times Q_{\gamma} \neq Q_{\gamma}$, and so we see that \textit{all non-compact imaginary roots are of type I}.  This establishes

\begin{prop}\label{prop:single_edges}
In the weak order graph for $K \backslash G/B$, all edges are single.
\end{prop}

\begin{remark}
The previous proposition follows from the discussion of the preceding paragraph, but can also be deduced using \cite[Corollary 2]{Brion-01}.  Indeed, this example is mentioned specifically in the discussion immediately following the statement of that corollary.
\end{remark}

Relative to the parametrization described here, the closed orbits (those minimal in the weak order) are those whose clans consist solely of $p$ plus signs and $q$ minus signs.  The dense open orbit is the one whose clan is $\gamma_0 := (1,2,\hdots,q-1,q,+,\hdots,+,q,q-1,\hdots,2,1)$ ($p-q$ plus signs appearing in the middle) if $p \geq q$, or $(1,2,\hdots,p-1,p,-,\hdots,-,p,p-1,\hdots,2,1)$ ($q-p$ minus signs appearing in the middle) if $q > p$.

With all of these combinatorics in hand, we recast the $M(W)$-action on $K$-orbits as a sequence of operations on $(p,q)$-clans.  Let $\gamma=(c_1,\hdots,c_n)$ be a $(p,q)$-clan.  Given a simple root $s_i = s_{\ga_i}$, consider the following two possible operations on $\gamma$:
\begin{enumerate}[(a)]
	\item Interchange characters $c_i$ and $c_{i+1}$.
	\item Replace characters $c_i$ and $c_{i+1}$ by matching natural numbers.
\end{enumerate}

Then for $i=1,\hdots,n-1$,
\begin{enumerate}
	\item If $c_i$ is a sign, $c_{i+1}$ is a natural number, and the mate of $c_{i+1}$ occurs to the right of $c_{i+1}$, then $s_i \cdot \gamma$ is obtained from $\gamma$ by operation (a).
	\item If $c_i$ is a number, $c_{i+1}$ is a sign, and the mate of $c_i$ occurs to the left of $c_i$, then $s_i \cdot \gamma$ is obtained from $\gamma$ by operation (a).
	\item If $c_i$ and $c_{i+1}$ are unequal natural numbers, with the mate of $c_i$ occurring to the left of the mate for $c_{i+1}$, then $s_i \cdot \gamma$ is obtained from $\gamma$ by operation (a).
	\item If $c_i$ and $c_{i+1}$ are opposite signs, then $s_i \cdot \gamma$ is obtained from $\gamma$ by operation (b).
	\item If none of the above hold, then $s_i \cdot \gamma = \gamma$.
\end{enumerate}

This extends in the obvious way to an action of $M(W)$ on the set of all $(p,q)$-clans.  Note that if $Y_{\gamma} = \overline{Q_{\gamma}}$ is a $K$-orbit closure, the geometric condition that $w \cdot Y_{\gamma} = G/B$ is equivalent to the combinatorial condition that $w \cdot \gamma = \gamma_0$.

\section{More on the combinatorics of $(p,q)$-clans}\label{sect:fs-patterns}
We shall be interested in the defining conditions (1) and (2) of the orbit $Q_{\gamma}$, given in Theorem \ref{thm:orbit_description}.  More specifically, we want to know to what extent a clan $\gamma$ can be determined by the numbers $\gamma(i; +)$ and $\gamma(i; -)$ ($i=1,\hdots,n$) alone.  To facilitate the discussion, we make the following definition:

\begin{definition}
Given a clan $\gamma=(c_1,\hdots,c_n)$, consider the following string of $n$ characters $(d_1,\hdots,d_n)$, each of which is a $+$, a $-$, an `F', or an `S':
\begin{itemize}
	\item If $c_i$ is a $+$ or a $-$, then $d_i = c_i$.
	\item If $c_i$ is the first occurrence of a given natural number, then $d_i$ is an `F'.
	\item If $c_i$ is the second occurrence of a given natural number, then $d_i$ is an `S'. 
\end{itemize}
We refer to the sequence $(d_1,\hdots,d_n)$ as the \textbf{first-second pattern} (or \textbf{FS-pattern}) for $\gamma$, and denote it by $FS(\gamma)$.
\end{definition}

For example, $FS((+,1,-,1)) = (+,F,-,S)$.  Also, 
\[ FS((1,2,3,3,2,1)) = FS((1,2,3,1,2,3)) = (F,F,F,S,S,S). \]
(Note that the clans $(1,2,3,3,2,1)$ and $(1,2,3,1,2,3)$ are different, but have the same FS-pattern.)

It is clear from the association of FS-patterns to $(p,q)$-clans that a sequence $(d_1,\hdots,d_n)$ of $\{+,-,F,S\}$ is the FS-pattern of at least one $(p,q)$-clan (for an appropriate $p,q$) if and only if the following two conditions hold:
\begin{enumerate}
	\item The number of occurrences of F is equal to the number of occurrences of S.
	\item For any $i \in [n]$, among $(d_1,\hdots,d_i)$, the number of occurrences of F is greater than or equal to the number of occurrences of S.
\end{enumerate}

The significance of FS-patterns is explained by the following 
\begin{lemma}\label{lem:fs-patterns}
Given any two $(p,q)$-clans $\gamma_1,\gamma_2$, the following are equivalent:
\begin{enumerate}
	\item $\gamma_1(i; +) = \gamma_2(i; +)$ and $\gamma_1(i; -) = \gamma_2(i; -)$ for all $i=1,\hdots,n$.
	\item $\gamma_1$, $\gamma_2$ have the same FS-pattern.
\end{enumerate}
\end{lemma}
\begin{proof}
Let $\gamma=(c_1,\hdots,c_n)$ be any $(p,q)$-clan.  Denote by $a_i$ the number $\gamma(i; +)$ and by $b_i$ the number $\gamma(i; -)$, for $i=1,\hdots,n$.  Defining $a_0 = b_0 = 0$, for any $i=1,\hdots,n$, one of the following four scenarios occurs:
\begin{enumerate}
	\item $a_i = a_{i-1} + 1$, $b_i = b_{i-1}$;
	\item $a_i = a_{i-1}$, $b_i = b_{i-1} + 1$;
	\item $a_i = a_{i-1}$, $b_i = b_{i-1}$;
	\item $a_i = a_{i-1} + 1$, $b_i = b_{i-1} + 1$.
\end{enumerate}

It is clear from the definitions of $\gamma(i; +)$ and $\gamma(i; -)$ that cases (1)-(4) occur precisely in the following situations:
\begin{enumerate}
	\item $c_i$ is a $+$;
	\item $c_i$ is a $-$;
	\item $c_i$ is the first occurrence of some natural number;
	\item $c_i$ is the second occurrence of some natural number.
\end{enumerate}

Thus for any clan $\gamma$, the collection of numbers $\gamma(i; +)$, $\gamma(i; -)$ ($i=1,\hdots,n$) and the FS-pattern $FS(\gamma)$ mutually determine one another, as we have claimed.
\end{proof}

We wish to see next that given any valid FS-pattern, the collection of $K$-orbits whose clans have that FS-pattern contains a unique maximal element in the (full) closure order.  For this, we recall the notion of pattern avoidance used in \cite{McGovern-Trapa-09}.

\begin{definition}
Given a $(p,q)$-clan $\gamma$ and a $(p',q')$-clan $\gamma'$ (with $p' \leq p$ and $q' \leq q$), $\gamma$ is said to \textbf{avoid the pattern} $\gamma'$ if there is no substring of $\gamma$ of length $p'+q'$ which is equal to $\gamma'$ as a clan.
\end{definition}

Among the set of clans with a given FS-pattern, we shall be interested in the unique one which avoids the pattern $(1,2,1,2)$.  This criterion, in plain English, says the following:  Moving from left to right, each time we encounter the second occurrence of a natural number, it is always the mate of the most recent as yet unmated natural number to appear.  For instance, $(1,2,2,3,3,1)$ avoids $(1,2,1,2)$, but $(1,2,2,3,1,3)$ does not, since the $1$ occurring in position 5 is the mate for the $1$ occurring in position $1$, while the $3$ occurring in position 4 is yet unmated.  We see the substring $(1,3,1,3)$ (which is equal, as a clan, to $(1,2,1,2)$) as a result.

It is easy to see that given a prescribed FS-pattern, there is a unique clan $\gamma$ with that FS-pattern which avoids $(1,2,1,2)$.  Indeed, given the FS-pattern $(d_1,\hdots,d_n)$, we write down such a clan $\gamma = (c_1,\hdots,c_n)$ as follows:  First, wherever $d_i = +$ or $d_i = -$, set $c_i = d_i$.  Next, for each $i$ such that $d_i = F$, make $c_i$ a distinct natural number.  (Assuming there are $m$ occurrences of $F$, these may as well be $1,\hdots,m$, in order from left to right.)  Finally, starting from the left and moving to the right, for each $i$ such that $d_i = S$, set $c_i$ to be the natural number equal to the closest as yet unmated natural number to the left.

As an example of the above, take the FS-pattern $(+,-,F,S,F,+,F,S,F,S,S)$.  Using the procedure above, we construct the clan $(+,-,1,1,2,+,3,3,4,4,2)$.  There are 3 other clans with this same FS-pattern, all of which include at least one instance of the pattern $(1,2,1,2)$.  These are 
\[ (+,-,1,1,2,+,3,3,4,2,4),(+,-,1,1,2,+,3,2,4,4,3),(+,-,1,1,2,+,3,2,4,3,4). \]

The key fact is that among the set of all orbits whose clans have a given FS-pattern, the one whose clan avoids $(1,2,1,2)$ is the unique maximal element in the full closure order.  To prove this, we need the following combinatorial fact regarding the full closure order on $K$-orbits:
\begin{lemma}
Suppose $\gamma=(c_1,\hdots,c_n)$ is a $(p,q)$-clan.  Suppose that $c_i$ and $c_j$ ($i<j$) are (not necessarily adjacent) unequal natural numbers such that the mate for $c_i$ occurs to the left of the mate for $c_j$.  Let $\gamma'$ be the clan obtained from $\gamma$ by interchanging $c_i$ and $c_j$.  Then $\overline{Q_{\gamma}} \subseteq \overline{Q_{\gamma'}}$.
\end{lemma}
\begin{proof}
This is stated, but not proven, in \cite{McGovern-Trapa-09}, so we give the proof here.  We will need a couple of properties of the full closure order on the orbits of any symmetric subgroup on the flag variety, which can be found in \cite{Richardson-Springer-90}.

The first property says the following.  Suppose there exists a codimension-1 subdiagram of the full closure order of the following form:
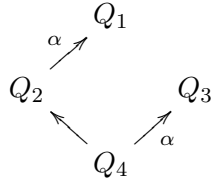
\[
\xymatrixcolsep{1pc} 
\xymatrixrowsep{1pc}
\xymatrix
{& Q_1 \\
Q_2 \ar@{>}[ur]^\alpha 
& & Q_3\\
& Q_4 \ar@{>}[ul] \ar@{>}[ur]_\alpha
}
\]
Then there must be an edge connecting $Q_3$ to $Q_1$:
\[
\xymatrixcolsep{1pc} 
\xymatrixrowsep{1pc}
\xymatrix
{& Q_1 \\
Q_2 \ar@{>}[ur]^\alpha 
& & Q_3\ar@{.>}[ul]\\
& Q_4 \ar@{>}[ul] \ar@{>}[ur]_\alpha
}
\]

This follows from \cite[Theorem 7.11, part (vii)]{Richardson-Springer-90}.  Indeed, as pointed out in \cite{McGovern-Trapa-09}, that result gives an algorithm for generating the full closure order from the weak closure order recursively, starting at the bottom of the weak ordering and applying this rule to add extra edges.  (We will not need this stronger result.)

Note that the lower-left edges in the diagrams above are unlabelled, meaning that they need not be weak order relations.  On the other hand, the upper-left and lower-right edges are required to be weak order relations \textit{with the same associated simple root}.  We also point out that in the general case, any of the edges in the above diagrams may be simple or double.  (We are dealing specifically with a case where all edges are simple, so this will not come up for us, but it is worth noting.)

For brevity, we refer to the property above as \textbf{Property A}.

The second property is as follows.  Suppose we have a codimension-1 subdiagram of the following form:
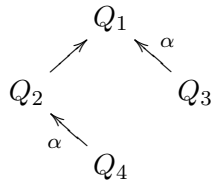
\[
\xymatrixcolsep{1pc} 
\xymatrixrowsep{1pc}
\xymatrix
{& Q_1 \\
Q_2 \ar@{>}[ur]
& & Q_3 \ar@{>}[ul]_\alpha \\
& Q_4 \ar@{>}[ul]^\alpha
}
\]

Then one of the following two scenarios must occur:
\begin{enumerate}
	\item There is an edge connecting $Q_4$ to $Q_3$:
	\[
	\xymatrixcolsep{1pc} 
	\xymatrixrowsep{1pc}
	
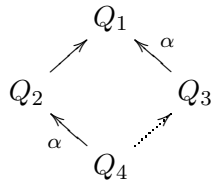
\xymatrix
	{& Q_1 \\
	Q_2 \ar@{>}[ur]
	& & Q_3 \ar@{>}[ul]_\alpha \\
	& Q_4 \ar@{>}[ul]^\alpha \ar@{.>}[ur]
	}
	\]
	\item There is another orbit $Q_5$, with $s_{\ga} \cdot Q_5 = Q_2$, and with an edge connecting $Q_5$ to $Q_3$:
	\[
	\xymatrixcolsep{1pc} 
	\xymatrixrowsep{1pc}
	\xymatrix
	{& Q_1 \\
	Q_2 \ar@{>}[ur]
	& & Q_3 \ar@{>}[ul]_\alpha \\
	& Q_4 \ar@{>}[ul]^\alpha & & Q_5 \ar@{>}[ulll]^\alpha \ar@{.>}[ul]
	}
	\]
\end{enumerate}

This property follows from \cite[Theorem 7.11, part (iv)]{Richardson-Springer-90}.  We refer to it as \textbf{Property B}.  Again, we remark that labelled edges are weak order relations, while unlabelled edges need not be.  Additionally, any of the edges may be either simple or double.

We now use these properties to prove the lemma.  There are two pairs of natural numbers in question, which could appear in either the pattern $(a,b,a,b)$ or $(a,a,b,b)$.  (The statement of the lemma does not apply to the pattern $(a,b,b,a)$.)  Suppose the numbers are in the pattern $(a,b,a,b)$.

Note that we have a ``choice" of interchanging the first $a$ and the first $b$ (resulting in $(b,a,a,b)$), or interchanging the second $a$ and the second $b$ (resulting in $(a,b,b,a)$).  But these two interchanges result in the same clan, so we may as well think of interchanging the first $a$ and the first $b$.  Suppose that the first $a$, the first $b$, the second $a$, and the second $b$ occur as characters $c_{i_1},c_{i_2},c_{i_3},c_{i_4}$ ($i_1 < i_2 < i_3 < i_4$), respectively.

Let $\gamma'$ be the clan obtained from $\gamma$ by interchanging the first $a$ and the first $b$.  We wish to prove that the orbit $Q_{\gamma'}$ is above $Q_{\gamma}$ in the full closure order.  The proof goes by induction on $i_2 - i_1$, i.e. on the distance between the numbers to be interchanged.  The case $i_2-i_1=1$ is handled by our discussion of the weak order in Section \ref{sect:pq_example}.  So suppose $i_2-i_1 > 1$.  There are a few cases to consider, depending on the value of $c_{i_1+1}$.

\textit{\textbf{Case 1:}  $c_{i_1+1} = \pm$} --- Then for
\[ \gamma = (\hdots,a,\pm,\hdots,b,\hdots,a,\hdots,b,\hdots); \gamma' = (\hdots,b,\pm,\hdots,a,\hdots,a,\hdots,b,\hdots); \]
\[ \delta = (\hdots,\pm,a,\hdots,b,\hdots,a,\hdots,b,\hdots); \epsilon =  (\hdots,\pm,b,\hdots,a,\hdots,a,\hdots,b,\hdots); \]
we have the following diagram:
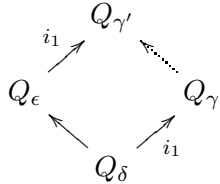
\[
\xymatrixcolsep{1pc} 
\xymatrixrowsep{1pc}
\xymatrix
{& Q_{\gamma'} \\
Q_{\epsilon} \ar@{>}[ur]^{i_1}
& & Q_{\gamma} \ar@{.>}[ul] \\
& Q_{\delta} \ar@{>}[ur]_{i_1} \ar@{>}[ul]
}
\]
The weak order edges in the upper-left and lower-right follow from the description of the weak order given in Section \ref{sect:pq_example}.  The lower-left edge is present by the inductive hypothesis.  Thus the upper-right edge must be present by Property A.

If $c_{i_1+1} =: c$ is a natural number, then there are three further cases to consider, depending on the position of the mate for $c$.  Let us say that the mate for the $c$ in position $i_1+1$ is in position $k$.

\textit{\textbf{Case 2:}  $k < i_3$} --- Then for 
\[ \gamma = (\hdots,a,c,\hdots,b,\hdots,a,\hdots,b,\hdots); \gamma' = (\hdots,b,c,\hdots,a,\hdots,a,\hdots,b,\hdots); \]
\[ \delta = (\hdots,c,a,\hdots,b,\hdots,a,\hdots,b,\hdots); \epsilon =  (\hdots,c,b,\hdots,a,\hdots,a,\hdots,b,\hdots); \]
we have the exact same diagram as above, for the exact same reasons.

\textit{\textbf{Case 3:}  $i_3 < k < i_4$} --- Then for 
\[ \gamma = (\hdots,a,c,\hdots,b,\hdots,a,\hdots,c,\hdots,b,\hdots); \gamma' = (\hdots,b,c,\hdots,a,\hdots,a,\hdots,c,\hdots,b,\hdots); \]
\[ \delta = (\hdots,c,a,\hdots,b,\hdots,a,\hdots,c,\hdots,b,\hdots); \epsilon =   (\hdots,c,b,\hdots,a,\hdots,a,\hdots,c,\hdots,b,\hdots); \]
we have the following diagram:
\[
\xymatrixcolsep{1pc} 
\xymatrixrowsep{1pc}
\xymatrix
{Q_{\gamma'} \\
Q_{\epsilon} \ar@{>}[u]_{i_1} \\
Q_{\delta} \ar@{>}[u] \\
Q_{\gamma} \ar@{>}[u]_{i_1}
}
\]
The top and bottom arrows follow by our description of the weak order, while the middle arrow follows by induction.

\textit{\textbf{Case 4:}  $i_4 < k$} --- For this case, let
\[ \gamma = (\hdots,a,c,\hdots,b,\hdots,a,\hdots,b,\hdots,c,\hdots); \gamma' = (\hdots,b,c,\hdots,a,\hdots,a,\hdots,b,\hdots,c,\hdots); \]
\[ \delta = (\hdots,c,a,\hdots,b,\hdots,a,\hdots,b,\hdots,c,\hdots); \epsilon =   (\hdots,c,b,\hdots,a,\hdots,a,\hdots,b,\hdots,c,\hdots); \]
Then we have the following diagram:
\[
\xymatrixcolsep{1pc} 
\xymatrixrowsep{1pc}
\xymatrix
{& Q_{\epsilon} \\
Q_{\delta} \ar@{>}[ur]
& & Q_{\gamma'} \ar@{>}[ul]_{i_1} \\
& Q_{\gamma} \ar@{>}[ul]^{i_1}
}
\]
The lower-left and upper-right arrows follow from our description of the weak order, while the upper-left arrow follows by induction.  Now, according to Property B, either there is an edge connecting $\gamma$ to $\gamma'$, or if not, then there is some other orbit $Q$ such that $s_{i_1} \cdot Q = Q_{\gamma'}$.  However, by our description of the weak order, it is clear that the only way to get to $\gamma'$ by $s_{i_1}$ is to come from $\gamma$.  (The only possibility other than switching $c_{i_1}$ and $c_{i_1+1}$ is to replace opposite signs in these positions by the \textit{same} natural number.  But since we have \textit{different} natural numbers occuring in positions $c_{i_1}$ and $c_{i_1+1}$, this cannot have occurred.)  By Property B, we conclude that there is an edge connecting $\gamma$ to $\gamma'$:
\[
\xymatrixcolsep{1pc} 
\xymatrixrowsep{1pc}
\xymatrix
{& Q_{\epsilon} \\
Q_{\delta} \ar@{>}[ur]
& & Q_{\gamma'} \ar@{>}[ul]_{i_1} \\
& Q_{\gamma} \ar@{>}[ul]^{i_1} \ar@{.>}[ur]
}
\]

This completes the proof in the event that our pair of matching natural numbers occurs in the pattern $(a,b,a,b)$.  If the numbers are in the pattern $(a,a,b,b)$, then one performs a nearly identical case-by-case analysis, with the same result.  Because these further cases are so nearly identical to the ones already treated above, we omit the details.
\end{proof}

Using the previous lemma, we now prove the following
\begin{prop}\label{prop:max_fs_orbit}
Let $(d_1,\hdots,d_n)$ be a valid FS-pattern.  Let $\{\gamma,\gamma_1,\hdots,\gamma_k\}$ be the set of clans having this FS-pattern, with $\gamma$ the unique one avoiding $(1,2,1,2)$.  Let $Q_{\gamma},Q_{\gamma_1},\hdots,Q_{\gamma_k}$ be the corresponding $K$-orbits.  Among these orbits, $Q_{\gamma}$ is the unique maximal one in the full closure order.  That is, for any $i=1,\hdots,k$, $\overline{Q_{\gamma_i}} \subseteq \overline{Q_{\gamma}}$.
\end{prop}
\begin{proof}
Suppose $\gamma_i \neq \gamma$ is given.  Then $\gamma_i$ contains at least one occurrence of the pattern $(1,2,1,2)$.  By the previous lemma, we know that when we interchange two of the unequal numbers, effectively ``removing" this occurrence (changing it to $(1,2,2,1)$), this results in a new orbit higher in the closure order.  Moreover, it is clear that the clan so obtained has the same FS-pattern as $\gamma_i$.  All we need to see, then, is that when we perform such an interchange, the number of occurrences of $(1,2,1,2)$ decreases.  If so, then clearly we can keep removing patterns of $(1,2,1,2)$ iteratively, moving higher in the closure order at each step, until eventually there are no patterns of $(1,2,1,2)$ left.  Because all of the clans obtained throughout this procedure have the same FS-pattern, the end result can only be $\gamma$.

It is not obvious on strictly combinatorial grounds that performing such an interchange must reduce the number of occurrences of $(1,2,1,2)$, for although the pattern in question is eliminated, new ones can be introduced which were not originally present.  For example, consider the clan $(1,2,3,1,4,2,3,4)$.  We see the pattern $(1,3,1,3)$, and interchange the second $1$ with the second $3$ to obtain the clan $(1,2,3,3,4,2,1,4)$.  Note that in the new clan, we see the pattern $(1,4,1,4)$, whereas in the original clan we had the pattern $(1,1,4,4)$.  However, the total number of occurrences of $(1,2,1,2)$ does decrease from $5$ ($(1,2,1,2),(1,3,1,3),(2,3,2,3),(2,4,2,4),(3,4,3,4)$) to $2$ ($(1,4,1,4),(2,4,2,4)$).

We can deduce that the number of occurrences of $(1,2,1,2)$ has to decrease using the previous lemma and the dimension formula (\ref{e:length}) given in Section \ref{sect:pq_example}.  Recall that the \textit{length} of the orbit $Q_{\gamma}$ is given as
\[ l(\gamma) = \sum_{c_i=c_j \in \bbN, i<j}\left( j-i-
\#\{k \in \bbN \; | \; c_s = c_t = k \text{ for some }
s < i<t<j \} \right ), \]
and that the dimension of $Q_{\gamma}$ is this length plus a constant.  An alternative way to write the formula for $l(\gamma)$ is
\[ l(\gamma) = \left[\sum_{c_i=c_j \in \bbN, i<j} (j-i) \right] - T, \]
where $T$ is the total number of occurrences of $(1,2,1,2)$ in $\gamma$.  Since changing an occurrence of $(1,2,1,2)$ to $(1,2,2,1)$ increases dimension, the length must increase.  But the sum in the formula above remains unchanged, since the $2$ is moved closer to its mate, while the $1$ is moved farther from its mate by the same amount.  Thus the only way the length can increase is for $T$ to decrease.
\end{proof}

\begin{example}
Starting with the clan $(1,2,3,1,4,2,3,4)$, we first eliminate the pattern $(1,3,1,3)$, as above, to obtain $(1,2,3,3,4,2,1,4)$.  Next, we eliminate $(1,4,1,4)$ to obtain $(1,2,3,3,4,2,4,1)$.  Finally, we eliminate the lone remaining $(1,2,1,2)$ occurrence, namely $(2,4,2,4)$, to arrive at $(1,2,3,3,4,4,2,1)$.  This last clan has no occurrences of $(1,2,1,2)$, and the corresponding $K$-orbit is maximal in the full closure order among orbits whose clans have the FS-pattern $(F,F,F,S,F,S,S,S)$. 
\end{example}

\section{The coincidence of $K$-orbit closures and Richardson varieties}
Using the combinatorics of the previous section, we now come to the key observation which links $K$-orbit closures and Theorem \ref{thm:brion} to Schubert calculus --- namely, that if $\gamma$ is a $(p,q)$-clan which avoids $(1,2,1,2)$, the closure of the orbit $Q_{\gamma}$ is a Richardson variety.

We start by defining some notation:
\begin{definition}
Given a $(p,q)$-clan $\gamma$ ($n=p+q$), define the following subsets of $[n]$:
\[ \gamma_+ := \{i \in [n] \ \vert \ c_i = + \text{ or } c_i \text{ is the second occurrence of a natural number} \}, \]
and
\[ \gamma_- := [n] - \gamma_+ = \{i \in [n] \ \vert \ c_i = - \text{ or } c_i \text{ is the first occurrence of a natural number} \}. \]

Similarly, define
\[ \widetilde{\gamma_+} := \{i \in [n] \ \vert \ c_i = + \text{ or } c_i \text{ is the first occurrence of a natural number} \}, \]
and 
\[ \widetilde{\gamma_-} := [n] - \widetilde{\gamma_+} = \{i \in [n] \ \vert \ c_i = - \text{ or } c_i \text{ is the second occurrence of a natural number}\}. \]
\end{definition}

\begin{definition}
Let $n=p+q$.  Define two functions $u$ and $v$ from the set of $(p,q)$-clans to $S_n$ as follows:

If $\gamma$ is a $(p,q)$-clan, the permutation $u(\gamma)$ is the one which assigns the numbers $p,p-1,\hdots,1$ to the elements of $\gamma_+$, in descending order, and the numbers $n,n-1,\hdots,p+1$ to the elements of $\gamma_-$, also in descending order.

The permutation $v(\gamma)$ is the one which assigns the numbers $1,\hdots,p$ to the elements of $\widetilde{\gamma_+}$, in ascending order, and the numbers $p+1,\hdots,n$ to $\widetilde{\gamma_-}$, also in ascending order.
\end{definition}

\begin{example}
For the $(3,3)$-clan $\gamma=(+,-,1,2,2,1)$, we have
\begin{itemize}
	\item $\gamma_+ = \{1,5,6\}$
	\item $\gamma_- = \{2,3,4\}$
	\item $u(\gamma) = 365421$ (one-line notation)
	\item $\widetilde{\gamma_+} = \{1,3,4\}$
	\item $\widetilde{\gamma_-} = \{2,5,6\}$
	\item $v(\gamma) = 142356$
\end{itemize}
\end{example}

With this notation defined, we now come to the first of our main results.
\begin{theorem}\label{thm:k-orbits-richardson-varieties}
Suppose $\gamma$ is a $(p,q)$-clan avoiding the pattern $(1,2,1,2)$.  Then $\overline{Q_{\gamma}}$ is the Richardson variety $X_{u(\gamma)}^{v(\gamma)}$.
\end{theorem}
\begin{proof}
Consider the following subset of $Fl(\C^n)$:
\[ Y_0 := \{F_{\bullet} \in Fl(\C^n) \ \vert \ \dim(F_i \cap E_p) = \gamma(i;+) \text{ and } \dim(F_i \cap \widetilde{E_q}) = \gamma(i;-) \ \forall i \in [n]\}. \]

Combining the linear algebraic description of $K$-orbits given in Theorem \ref{thm:orbit_description} with Lemma \ref{lem:fs-patterns}, it follows that $Y_0$ is precisely the union $Z$ of all $K$-orbits $Q_{\lambda}$ where $FS(\lambda) = FS(\gamma)$.  Indeed, it follows from Lemma \ref{lem:fs-patterns} that $Z \subset Y_0$, since all clans $\lambda$ with $FS(\lambda) = FS(\gamma)$ have the property that $\lambda(i;+) = \gamma(i;+)$ and $\lambda(i;-) = \gamma(i;-)$ for $i=1,\hdots,n$.  The opposite inclusion follows from Lemma \ref{lem:fs-patterns} combined with Theorem \ref{thm:orbit_description} --- any flag $F_{\bullet} \in Y_0$ must be in \textit{some} $K$-orbit, and that $K$-orbit must be $Q_{\lambda}$ for some $\lambda$ by Theorem \ref{thm:orbit_description}.  By Lemma \ref{lem:fs-patterns} again, $\lambda$ must have the same FS-pattern as $\gamma$, or else $F_{\bullet}$ would satisfy different incidence conditions with $E_p$ and $\widetilde{E_q}$ than those which define $Y_0$.

Since $Y_0 = Z$, it follows from Proposition \ref{prop:max_fs_orbit} that $Y := \overline{Y_0}$ is precisely $\overline{Q_{\gamma}}$.  Thus we have only to show that $Y$ is the Richardson variety $X_{u(\gamma)}^{v(\gamma)}$.

Write $Y_0 = Y_0^+ \cap Y_0^-$, where
\[ Y_0^+ := \{F_{\bullet} \in Fl(\C^n) \ \vert \ \dim(F_i \cap E_p) = \gamma(i;+) \ \forall i \in [n]\}, \]
and
\[ Y_0^- := \{F_{\bullet} \in Fl(\C^n) \ \vert \ \dim(F_i \cap \widetilde{E_q}) = \gamma(i;-) \ \forall i \in [n]\}. \]

Based on the linear algebraic description of type $A$ Schubert cells $C_w$ given in Definition \ref{def:schubert_cells_varieties}, it is immediate that $Y_0^+$ is a union of Schubert cells.  Indeed, define 
\[ W^+ := \{ w \in W \ \vert \ r_w(i,p) = \gamma(i;+) \ \forall i \in [n]\}. \]
Then
\[ Y_0^+ = \bigcup_{w \in W^+} C_w. \]
(One sees this by a very similar argument to the one above regarding the equality of $Z$ and $Y_0$.)

Similarly (cf. Definition \ref{def:opposite_cells_varieties}), $Y_0^-$ is a union of opposite Schubert cells.  Namely,
\[ Y_0^- = \bigcup_{w \in W^-} C^w, \]
where
\[ W^- := \{w \in W \ \vert \ r_{w_0w}(i,q) = \gamma(i;-) \ \forall i \in [n]\}. \]

Thus $Y_0$ is a union of sets $C_a^b$ with $a \in W^+$ and $b \in W^-$.  To show that $Y$ is the Richardson variety $X_{u(\gamma)}^{v(\gamma)}$, it will suffice to show that precisely one of these sets --- specifically $C_{u(\gamma)}^{v(\gamma)}$ --- is maximal in the closure order.  By Proposition \ref{prop:richardson_closure_order}, this is equivalent to showing that $u(\gamma)$ is the unique maximal element of $W^+$, and that $v(\gamma)$ is the unique minimal element of $W^-$.

Consider permutations $w \in W^+$.  They are precisely those permutations whose rank matrices have $p$th column prescribed by the numbers $\gamma(i;+)$:
\[
\begin{pmatrix}
	r_w(1,1) & \hdots & r_w(1,p) & \hdots & r_w(1,n) \\
	\vdots & \vdots & \vdots & \vdots & \vdots \\
	r_w(n,1) & \hdots & r_w(n,p) & \hdots & r_w(n,n)
\end{pmatrix}
=
\begin{pmatrix}
	* & \hdots & \gamma(1;+) & \hdots & * \\
	\vdots & \vdots & \vdots & \vdots & \vdots \\
	* & \hdots & \gamma(n;+) & \hdots & *
\end{pmatrix}
\]

By Definition \ref{def:bruhat-1} of the Bruhat order, to see that this set contains a unique maximal element, it suffices to show that the remaining entries of the rank matrix can be ``filled in" in a way which produces a rank matrix $R$ such that any other rank matrix having the prescribed $p$th column must be greater than or equal than $R$ in every single position.

The way to accomplish this is to place the jumps as far to the right as possible on every single row, starting with the first.  Set $\gamma(0;+) = 0$.  Then for any $i \in [n]$, either $\gamma(i;+) = \gamma(i-1;+)$, or $\gamma(i;+) = \gamma(i-1;+) + 1$.

If $\gamma(i;+) = \gamma(i-1;+)$, the jump in the $i$th row has not yet occurred by the point we reach the $p$th column.  We put it as far to the right as possible, meaning that the first time we encounter such a row, we place the jump in position $n$, the second time we put it in position $n-1$, etc.

If $\gamma(i;+) = \gamma(i-1;+) + 1$, then the jump in the $i$th row \textit{has} occurred by the $p$th column.  Again, we want to place the jump as far to the right as possible, so the first time we encounter such a row, we put the jump at $p$, the second time at $p-1$, etc.

This gives us the rank matrix of a permutation which assigns the numbers $n,n-1,\hdots,p+1$ to those $i$ with $\gamma(i;+) = \gamma(i-1;+)$, in descending order, and which assigns the numbers $p,p-1,\hdots,1$ to those $i$ with $\gamma(i;+) = \gamma(i-1;+) + 1$, also in descending order.  Note that the $i$ of the former type are the elements of $\gamma^-$, while $i$ of the latter type are the elements of $\gamma^+$.  This shows that $u(\gamma)$ is in fact the unique maximal element of $W^+$.

As an example of the above, consider the $(2,2)$-clan $(+,1,-,1)$.  This clan prescribes the second column of a rank matrix as follows:
\[
\begin{pmatrix}
	* & 1 & * & * \\
	* & 1 & * & * \\
	* & 1 & * & * \\
	* & 2 & * & *
\end{pmatrix}
\]

The jump in the first row has already occurred by position $2$, since there is a $1$ there.  So we make the upper-left entry as small as possible by setting it to $0$.  The remaining entries of the first row are then determined:
\[
\begin{pmatrix}
	0 & 1 & 1 & 1 \\
	* & 1 & * & * \\
	* & 1 & * & * \\
	* & 2 & * & *
\end{pmatrix}
\]

Now we move on to the second row.  The jump in the second row has not occurred by the second column, since that entry is still $1$.  So we make the entries of the second row as small as possible by allowing them to be equal to the corresponding entries of the first row for as long as possible.  This forces the second jump all the way to the right:
\[
\begin{pmatrix}
	0 & 1 & 1 & 1 \\
	0 & 1 & 1 & 2 \\
	* & 1 & * & * \\
	* & 2 & * & *
\end{pmatrix}
\]

The third row is similar, although now we must jump in the third column, since there is already a jump in the fourth column:
\[
\begin{pmatrix}
	0 & 1 & 1 & 1 \\
	0 & 1 & 1 & 2 \\
	0 & 1 & 2 & 3 \\
	* & 2 & * & *
\end{pmatrix}
\]

Finally, the jump in the fourth row has no choice but to occur in position $1$:
\[
\begin{pmatrix}
	0 & 1 & 1 & 1 \\
	0 & 1 & 1 & 2 \\
	0 & 1 & 2 & 3 \\
	1 & 2 & 3 & 4
\end{pmatrix}
\]

The resulting rank matrix has jumps in positions $2$, $4$, $3$, and $1$, so it is the rank matrix of the permutation $2431$.  Note that $2$ and $1$ are assigned to the $+$ and the second occurrence of the $1$ in $(+,1,-,1)$ (the coordinates of $\gamma^+$), while $4$ and $3$ are assigned to the $-$ and the first occurrence of the $1$ (the coordinates of $\gamma^-$).

To find the unique minimal element of $W^-$, first find the unique maximal element $w_0w$ whose rank matrix has $q$th column prescribed by the numbers $\gamma(i;-)$.  We do this exactly as above, obtaining a permutation which assigns $n,n-1,\hdots,q+1$ to elements of $\widetilde{\gamma^+}$, and which assigns $q,q-1,\hdots,1$ to the elements of $\widetilde{\gamma^-}$.  Multiplying the resulting permutation by $w_0$, we obtain $v(\gamma)$.
\end{proof}

\section{A positive description of $c_{u,v}^w$ for $(p,q)$-pairs $(u,v)$}
Note that the permutations $u=w_0 \cdot u(\gamma)$, $v=v(\gamma)$ produced in the proof of Theorem \ref{thm:k-orbits-richardson-varieties} have the following properties:
\[ u^{-1}(1) < u^{-1}(2) < \hdots < u^{-1}(q) \text{ and } u^{-1}(q+1) < u^{-1}(q+2) < \hdots < u^{-1}(n); \]
\[ v^{-1}(1) < v^{-1}(2) < \hdots < v^{-1}(p) \text{ and } v^{-1}(p+1) < v^{-1}(p+2) < \hdots < v^{-1}(n). \]
Said another way, the one-line notation for $u$ is a ``shuffle" of $1,\hdots,q$ and $q+1,\hdots,n$ (since $u(\gamma)$ is a shuffle of $p,\hdots,1$ and $n,n-1,\hdots,p+1$), while the one-line notation for $v$ is a shuffle of $1,\hdots,p$ and $p+1,\hdots,n$.  This motivates the following definition:

\begin{definition}
For $p+q=n$, suppose that $u,v \in S_n$ are two permutations having the following properties:
\begin{enumerate}
	\item $u$ is a shuffle of $1,\hdots,q$ and $q+1,\hdots,n$.
	\item $v$ is a shuffle of $1,\hdots,p$ and $p+1,\hdots,n$.
\end{enumerate}

Then we call $(u,v)$ a \textbf{$(p,q)$-pair}.
\end{definition}

Taken together, Theorem \ref{thm:k-orbits-richardson-varieties}, Theorem \ref{thm:brion}, and the combinatorics laid out in Section \ref{sect:pq_example} give a positive (indeed, multiplicity-free) rule for structure constants $c_{u,v}^w$ when $(u,v)$ is a $(p,q)$-pair of the form $(w_0 \cdot u(\gamma),v(\gamma))$, with $\gamma$ a $(p,q)$-clan avoiding the pattern $(1,2,1,2)$.  The hope would be that this rule applies to \textit{any} $c_{u,v}^w$ with $(u,v)$ a $(p,q)$-pair.  Of course, if $w_0u$ and $v$ are not comparable in the Bruhat order, then $c_{u,v}^w$ is automatically zero, so it suffices to consider only those $(p,q)$-pairs $(u,v)$ with $w_0u$ and $v$ comparable.  The next proposition states that, indeed, \textit{any} such $(p,q)$-pair is of the form $(w_0 \cdot u(\gamma),v(\gamma))$ for some $(p,q)$-clan $\gamma$.

\begin{prop}\label{prop:any_pq_pair_is_clan}
Suppose $(u,v)$ is a $(p,q)$-pair, with $w_0u \geq v$.  Then there exists a $(p,q)$-clan $\gamma$ avoiding the pattern $(1,2,1,2)$ such that $w_0u = u(\gamma)$ and $v = v(\gamma)$.
\end{prop}
\begin{proof}
To be clear, the question here concerns whether any pair of comparable permutations, one of which is a shuffle of $p,p-1,\hdots,1$ and $n,n-1,\hdots,p+1$ (this is $w_0u$), and the other of which is a shuffle of $1,\hdots,p$ and $p+1,\hdots,n$ (this is $v$), always arises as $u(\gamma)$ and $v(\gamma)$ for some $(p,q)$-clan $\gamma$.  For the purpose of avoiding any confusion between $u$ and $w_0u$, let us just say that $a$ is a shuffle of $p,p-1,\hdots,1$ and $n,n-1,\hdots,p+1$, and that $b$ is a shuffle of $1,\hdots,p$ and $p+1,\hdots,n$.  Define the \textbf{high-low pattern} for the pair $(a,b)$ to be a sequence $(e_1,\hdots,e_n)$ of $n$ symbols, each of which is a $+$, $-$, H, or L.  For each $i$, the value of $e_i$ is determined as follows:
\begin{enumerate}
	\item If $a(i),b(i) \leq p$, $e_i = +$.
	\item If $a(i),b(i) > p$, $e_i = -$.
	\item If $a(i) > p$ and $b(i) \leq p$, $e_i = H$.
	\item If $a(i) \leq p$ and $b(i) > p$, $e_i = L$.
\end{enumerate}

Comparability of $a$ and $b$ depends on the positioning of H's and L's, in the following way:  For any $i \in [n]$, denote by $(i;H)$ (resp., $(i;L),(i;+),(i;-)$) the number of H's (resp., the number of L's, $+$'s, and $-$'s) occurring among $(e_1,\hdots,e_i)$.  Then for the pair $(a,b)$, we have that $a \geq b$ if and only if $(i;H) \geq (i;L)$ for all $i \in [n]$.

To see this, we use the characterization of the Bruhat order given in Definition \ref{def:bruhat-2}.  Since $a$ is a shuffle of $p,p-1,\hdots,1$ and $n,n-1,\hdots,p+1$, for any $i \in [n]$, when $\{a(1),\hdots,a(i)\}$ is arranged in ascending order, the result is of the form
\[ \{s,s+1,\hdots,p,\vert,t,t+1,\hdots,n\}, \]
for some $s \leq p$ and for some $t > p$.  (The vertical bar marks the point at which the values change from being less than or equal to $p$ to being greater than $p$.)

Similarly, since $b$ is a shuffle of $1,\hdots,p$ and $p+1,\hdots,n$, when $\{b(1),\hdots,b(i)\}$ is arranged in ascending order, the result is
\[ \{1,2,\hdots,h,\vert,p+1,p+2,\hdots,k\}, \]
for some $h \leq p$ and $k > p$.

Comparing these sets, it is clear that the second set is element-wise less than or equal to the first if and only if the second set has at least as many elements which are less than or equal to $p$ as the first set does.  (Said another way, the vertical bar in the second set appears at least as far to the right as the vertical bar in the first set does.)  The number of elements to the left of the vertical bar in the second set is $(i;+) + (i;H)$, while the the number elements to the left of the vertical bar in the first set is $(i;+)+(i;L)$.  And $(i;+)+(i;H) \geq (i;+)+(i;L)$ if and only if $(i;H) \geq (i;L)$.  This proves the claim.

Note that it is obvious that the high-low pattern for \textit{any} pair of permutations whatsoever has an equal number of H's and L's.  Indeed, to say otherwise would say that one permutation had more values from either the set $\{1,\hdots,p\}$ or the set $\{p+1,\hdots,n\}$ than the other, which is clearly absurd.

So the high-low pattern for $(a,b)$ has the following two properties:
\begin{enumerate}
	\item The number of occurrences of H is equal to the number of occurrences of L.
	\item For any $i \in [n]$, among $(e_1,\hdots,e_i)$, the number of occurrences of H is greater than or equal to the number of occurrences of L.
\end{enumerate}

Recall that the properties of a valid FS-pattern (cf. Section \ref{sect:fs-patterns}) are identical to these, with H's replaced by F's, and L's replaced by S's.  So the high-low pattern for the pair $(a,b)$ corresponds to an FS-pattern in the most obvious way possible:  Simply change the H's to F's, and the L's to S's.  Let $\gamma$ be the unique $(p,q)$-clan avoiding $(1,2,1,2)$ with the resulting FS-pattern.  It is now clear by construction that $a = u(\gamma)$ and $b = v(\gamma)$, since $u(\gamma)$ is defined by putting $p,p-1,\hdots,1$ on the $+$'s and S's, and $n,n-1,\hdots,p+1$ on the $-$'s and F's, while $v(\gamma)$ is defined by putting $1,\hdots,p$ on the $+$'s and F's, and $p+1,\hdots,n$ on the $-$'s and S's.
\end{proof}                                               

\begin{definition}\label{def:pq_pair}
If $(u,v)$ is a $(p,q)$-pair, we denote by \textbf{$\gamma(u,v)$} the unique clan $\gamma$ avoiding $(1,2,1,2)$ and having the property that $w_0u=u(\gamma)$, $v = v(\gamma)$.
\end{definition}

\begin{example}
Consider the $(4,4)$-clan $\gamma=(+,1,1,2,3,-,3,2)$.  We know that $Y_{\gamma}$ is the Richardson variety $X_{u(\gamma)}^{v(\gamma)}$, with
\[ u(\gamma) = 48376521, \]
and
\[ v(\gamma) = 12534678. \]

This corresponds to the $(4,4)$-pair $(u,v)$, with $u=w_0 \cdot u(\gamma) = 51623478$, and $v = v(\gamma)$.  Suppose that we were instead handed this $(4,4)$-pair with no prior knowledge of $\gamma$.  Then the proof of Proposition \ref{prop:any_pq_pair_is_clan} says that we should first construct the high-low pattern for $u(\gamma)$, $v(\gamma)$, which is $(+,H,L,H,H,-,L,L)$.  This is then converted to the FS-pattern $(+,F,S,F,F,-,S,S)$.  From this FS-pattern, we recover $\gamma=(+,1,1,2,3,-,3,2)$.  (See Section \ref{sect:fs-patterns}.)
\end{example}

We can now put the various pieces together to arrive at our main result:
\begin{theorem}\label{thm:structure_constants}
Let $p+q=n$.  Let $\gamma_0$ be the clan parametrizing the open dense $GL(p,\C) \times GL(q,\C)$-orbit on $G/B$, as in Section \ref{sect:pq_example}.  Suppose that $(u,v)$ is a $(p,q)$-pair, with $w_0u \geq v$.  Then 
\[ 
c_{u,v}^w = 
\begin{cases}
	1 & \text{ if $l(w) = l(u) + l(v)$ and $w \cdot \gamma(u,v) = \gamma_0$}, \\
	0 & \text{ otherwise.}
\end{cases}
\]
\end{theorem}
\begin{proof}
Recall that $\gamma_0$ is the clan $(1,2,\hdots,q-1,q,+,\hdots,+,q,q-1,\hdots,2,1)$ ($p-q$ plus signs in the middle) if $p \geq q$, and $(1,2,\hdots,p-1,p,-,\hdots,-,p,p-1,\hdots,2,1)$ ($q-p$ minus signs) if $q > p$.

Let $Y=\overline{Q_{\gamma(u,v)}}$.  By Theorem \ref{thm:k-orbits-richardson-varieties}, along with equation (\ref{eqn:classes_of_richardson_varieties}) of Subsection \ref{sect:schubert_defs}, we have 
\[ [Y] = [X_{w_0u}^{v}] = S_u \cdot S_v, \]
so the structure constants $c_{u,v}^w$ are identically the coefficients of the various $S_w$ in the Schubert basis expansion of $[Y]$.

The fact that all such coefficients are $0$ or $1$ follows from Proposition \ref{prop:single_edges} and Theorem \ref{thm:brion}.  Note that requiring that $l(w) = l(u) + l(v)$ (hence requiring that $S_w$ live in the only degree it could in order for $c_{u,v}^w$ to be non-zero) is equivalent to requiring that $l(w) = \text{codim}(Y)$, as we do in the definition of $W(Y)$ prior to the statement of Theorem \ref{thm:brion}.  Indeed, the codimension of $Y$ is precisely
\[ \dim(G/B) - \dim(X_{w_0u}^v) = \]
\[ \dim(G/B) - (l(w_0u) - l(v)) = \]
\[ (\dim(G/B) - l(w_0u)) + l(v) = \]
\[  l(u) + l(v). \]
Thus by Theorem \ref{thm:brion}, the only other requirement we must impose on $w$ for $c_{u,v}^w = 1$ is that $w \cdot Y = G/B$.  As was noted at the end of Section \ref{sect:pq_example}, this is equivalent to the combinatorial condition that $w \cdot \gamma(u,v)=\gamma_0$.
\end{proof}

\begin{example}\label{ex:example-1}
Consider the $(3,2)$-pair $(u,v) = (31425,14253)$.  The Schubert product $S_u \cdot S_v$ corresponds to the $(3,2)$-clan $\gamma(u,v) = (+,-,+,-,+)$.  We have $l(u) = l(v) = 3$, and there are $20$ elements of $S_5$ of length $l(u) + l(v) = 6$.  Table 1 of the Appendix shows each of these $20$ elements as words in the simple reflections, the clan obtained from computing the action of each on the clan $\gamma(u,v)$, and the corresponding structure constant $c_{u,v}^w$ specified by Theorem \ref{thm:structure_constants}.

The data in Table 1, obtained using Theorem \ref{thm:structure_constants}, was checked against the output of Maple code written by A. Yong (\cite{Yong-Maple}), and was found to agree with the output of that program.

\end{example}

\section{A Final Question}
Theorem \ref{thm:brion} applies very generally to the class of any spherical subgroup orbit closure in any flag variety.  (Indeed, as mentioned in the introduction, it is applied in the paper \cite{Wyser-11b} to obtain type $CD$ analogues of Theorem \ref{thm:structure_constants}.)  We feel it is natural to wonder whether there are other cases in which Theorem \ref{thm:brion} can be used to obtain information on type $A$ Schubert calculus.  We leave the reader with this question.

\begin{question}
Are there other examples of spherical subgroups of $GL(n,\C)$, the closures of whose orbits coincide with Richardson varieties?  If so, are combinatorial parametrizations of those orbits understood, and is the $M(W)$-action on the orbits understood on the level of that combinatorial parametrization?
\end{question}

\newpage

\appendix
\section{Tables for Examples}
\begin{table}[h]
	\caption{Example \ref{ex:example-1}:  Computing the $(3,2)$ Schubert product $S_{31425} \cdot S_{14253}$}
	\begin{tabular}{|c|c|c|}
		\hline
		Length $6$ Element $w$ & $w \cdot (+,-,+,-,+)$ & $c_{u,v}^w$ \\ \hline
		$[4, 3, 2, 4, 3, 4]$ & $(+,1,2,2,1)$  & $0$ \\ \hline
		$[1, 3, 2, 4, 3, 4]$ & $(1,+,2,2,1)$ & $0$ \\ \hline
		$[1, 4, 3, 2, 3, 4]$ & $(1,+,2,2,1)$ & $0$ \\ \hline
		$[1, 4, 3, 2, 4, 3]$ & $(1,+,2,2,1)$ & $0$ \\ \hline
		$[2, 1, 2, 4, 3, 4]$ & $(1,2,2,+,1)$ & $0$ \\ \hline
		$[2, 1, 3, 2, 3, 4]$ & $(1,2,+,2,1)$ & $1$ \\ \hline
		$[2, 1, 3, 2, 4, 3]$ & $(1,2,+,2,1)$ & $1$ \\ \hline
		$[2, 1, 4, 3, 2, 4]$ & $(1,2,+,2,1)$ & $1$ \\ \hline
		$[2, 1, 4, 3, 2, 3]$ & $(1,2,2,+,1)$ & $0$ \\ \hline
		$[3, 2, 1, 4, 3, 4]$ & $(1,2,+,1,2)$ & $0$ \\ \hline
		$[3, 2, 1, 2, 3, 4]$ & $(1,2,+,2,1)$ & $1$\\ \hline
		$[3, 2, 1, 2, 4, 3]$ & $(1,2,+,2,1)$ & $1$ \\ \hline
		$[3, 2, 1, 3, 2, 4]$ & $(1,2,+,1,2)$ & $0$ \\ \hline
		$[3, 2, 1, 3, 2, 3]$ & $(1,2,2,1,+)$ & $0$ \\ \hline
		$[3, 2, 1, 4, 3, 2]$ & $(1,2,+,2,1)$ & $1$ \\ \hline
		$[4, 3, 2, 1, 3, 4]$ & $(1,2,+,2,1)$ & $1$  \\ \hline
		$[4, 3, 2, 1, 4, 3]$ & $(1,2,+,2,1)$ & $1$ \\ \hline
		$[4, 3, 2, 1, 2, 4]$ & $(1,+,2,2,1)$ & $0$ \\ \hline
		$[4, 3, 2, 1, 2, 3]$ & $(1,2,2,+,1)$ & $0$\\ \hline
		$[4, 3, 2, 1, 3, 2]$ & $(1,2,2,+,1)$ & $0$ \\ \hline
	\end{tabular}
\end{table}

\bibliographystyle{alpha}
\bibliography{sourceDatabase}

\end{document}